\documentclass[12pt]{amsart}
\usepackage{amsfonts}
\usepackage{amssymb}
\usepackage[letterpaper, left=2.5cm, right=2.5cm, top=2.5cm,
bottom=2.5cm,dvips]{geometry}
\usepackage{verbatim}
\usepackage[T1]{fontenc}
\usepackage{graphicx}
\usepackage{amsmath}

\newtheorem{theorem}{Theorem}
\theoremstyle{plain}

\newtheorem{case}{Case}
\newtheorem{claim}{Claim}

\newtheorem{conjecture}{Conjecture}

\newtheorem{problem}{Problem}

\newtheorem{dupa}{Subcase}
\numberwithin{equation}{section}

\usepackage{array,multirow}

\usepackage{xcolor}

\usepackage[backref]{hyperref}
\hypersetup{
	colorlinks,
	linkcolor={blue!60!black},
	citecolor={red!60!black},
	urlcolor={red!60!black}
}

\begin{document}
	\title{Extensions and reductions of square-free words}

	\author{Micha\l{} D\k{E}bski}
	\address{Faculty of Mathematics and Information Science, Warsaw University
		of Technology, 00-662 Warsaw, Poland}
	\email{m.debski@mini.pw.edu.pl}
	
	\author{Jaros\l aw Grytczuk}
	\address{Faculty of Mathematics and Information Science, Warsaw University
		of Technology, 00-662 Warsaw, Poland}
	\email{j.grytczuk@mini.pw.edu.pl}
	
	\author{Bart\l omiej Pawlik}
	\address{Institute of Mathematics, Silesian University of Technology, 44-100 Gliwice, Poland}
	\email{bpawlik@polsl.pl}

	\begin{abstract}
A word is \emph{square-free} if it does not contain a nonempty word of the form $XX$ as a factor. A famous 1906 result of Thue asserts that there exist arbitrarily long square-free words over a $3$-letter alphabet. We study square-free words with additional properties involving single-letter deletions and extensions of words.

A square-free word is \emph{steady} if it remains square-free after deletion of any single letter. We prove that there exist infinitely many steady words over a $4$-letter alphabet. We also demonstrate that one may construct steady words of any length by picking letters from arbitrary alphabets of size $7$ assigned to the positions of the constructed word. We conjecture that both bounds can be lowered to $4$, which is best possible.

In the opposite direction, we consider square-free words that remain square-free after insertion of a single (suitably chosen) letter at every possible position in the word. We call them \emph{bifurcate}. We prove a somewhat surprising fact, that over a fixed alphabet with at least three letters, every steady word is bifurcate. We also consider families of bifurcate words possessing a natural tree structure. In particular, we prove that there exists an infinite tree of doubly infinite bifurcate words over alphabet of size $12$.

	\end{abstract}
	
	\maketitle
	
	\section{Introduction}
	A \emph{square} is a nonempty word of the form $XX$. A word $W$ \emph{contains} a square if it can be written as $W=PXXS$, with $X$ nonempty, while $P$ and $S$ possibly empty. Otherwise, $W$ is called \emph{square-free}. A famous theorem of Thue \cite{Thue} (see \cite{BerstelThue}) asserts that there exist infinitely many square-free words over a $3$-letter alphabet. This result initiated combinatorics on words --- the whole branch of mathematics, abundant in many deep results, exciting problems, and unexpected connections to diverse areas of science (see \cite{AlloucheShallit,BeanEM,BerstelPerrin,Currie TCS,Lothaire,LothaireAlgebraic}).
	
	In this paper we study square-free words with additional properties involving two natural operations, a single-letter extension and a single-letter deletion, defined as follows.
	
	Let $W$ be a word of length $n$ over a fixed alphabet $\mathcal{A}$. For $i=0,1,\dots,n$, let $P_i(W)$ and $S_i(W)$ denote the prefix and the suffix of $W$ of length $i$, respectively. Notice that the~words $P_0(W)$ and $S_0(W)$ are empty. In particular, we have $W=P_i(W)S_{n-i}(W)$ for every $i=0,1,\dots,n$. An \emph{extension} of $W$ at \emph{position} $i$ is a word of the form $U=P_i(W)\mathtt{x}S_{n-i}(W)$, where $\mathtt{x}$ is any letter from $\mathcal{A}$. In this case we also say that $W$ is a \emph{reduction} of $U$.
	
	A square-free word $W$ is \emph{extremal} over $\mathcal{A}$ if there is no square-free extension of $W$. Grytczuk, Kordulewski and Niewiadomski \cite{GrytczukKN} proved that there exist infinitely many extremal ternary words, the shortest one being $$\mathtt{1231213231232123121323123}.$$They also conjectured that there are no extremal words over an alphabet of size $4$. Recently, Hong and Zhang \cite{HongZhang} proved that this is true for an alphabet of size $17$.
	
	Harju \cite{Harju} introduced a complementary concept of \emph{irreducible} words. These are square-free words whose any non-trivial reduction (the deleted letter is neither the first one nor the~last one in the word) contains a square. He proved that for any $n\neq4,5,7,12$ there exists a~ternary irreducible word of length $n$.
	
	In this article we consider square-free words with the very opposite properties, defined as follows: a square-free word is \emph{steady} if it remains square-free after deleting any single letter. For instance, the word $\mathtt{1231}$ is steady since all of its four reductions $$\mathtt{231,131,121,231}$$are square-free, while $\mathtt{1213}$ is not steady since one of its reductions is $\mathtt{113}$. Generally, every ternary square-free word of length at least $6$ must contain a factor of the form $\mathtt{xyx}$, and therefore is not steady. However, there are steady words of any given length over larger alphabets, as we prove in Theorem \ref{Theorem Steady 4}. We also consider a general variant of such statement in the following \emph{list} setting.
	
	\begin{conjecture}\label{Conjecture 4-List Steady}
		Let $n$ be a positive integer and let $\mathcal{A}_1,\mathcal{A}_2,\ldots,\mathcal{A}_n$ be a sequence of alphabets of size $4$. Then there exists a steady word $W=w_1w_2\cdots w_n$, with $w_i\in\mathcal{A}_i$ for every $i=1,2,\ldots,n$.
	\end{conjecture}
	
	In Theorem \ref{Theorem 7-list Steady} we prove that the statement of the conjecture holds for alphabets with at least $7$ letters. Let us mention that analogous conjecture for pure square-free words (with alphabets of size $3$), stated in \cite{Grytczuk}, is still open. Currently the best general result confirms it for alphabets of size $4$ (see \cite{GrytczukKM,GrytczukPZ,Rosenfeld1} for three different proofs). Recently Rosenfeld \cite{Rosenfeld2} proved that it holds when the union of all alphabets is a $4$-element set.
	
	We also consider square-free words defined similarly with respect to extensions of words. A square-free word is \emph{bifurcate} over a fixed alphabet $\mathcal{A}$ if it has at least one square-free extension at \emph{every} position. For instance, the word $\mathtt{1231}$ is bifurcate over $\{1,2,3\}$ and here are its five square-free extensions: $$\mathtt{\underline{2}1231,1\underline{3}231,12\underline{1}31,123\underline{2}1,1231\underline{2}}.$$
	Thus the word $\mathtt{1231}$ is both steady and bifurcate. This not a coincidence --- for an alphabet with at least three letters, every steady word is bifurcate, as we prove in Theorem \ref{Theorem Steady-Bifurcate}.
	
	Clearly, ternary bifurcate words cannot be too long. Indeed, every ternary square-free word of length at least $6$ contains a factor of the form $\mathtt{xyxz}$ (or its reversal). On the other hand, any ternary square-free word is bifurcate over a $4$-letter alphabet. One may, however, inquire about the existence of an infinite \emph{chain} of bifurcate quaternary words.
	
	\begin{conjecture}\label{Conjecture 4-Bifurcate Chain}
		There exists an infinite sequence of quaternary bifurcate words $W_1,W_2,\ldots$ such that $W_{i+1}$ is a single-letter extension of $W_i$, for each $i=1,2,\ldots$.
	\end{conjecture}
	
A much stronger property holds over alphabets of size at least $12$, namely there exists a~\emph{complete bifurcate tree} of bifurcate words (rooted at any single letter), in which every word of length $n$ has $n+1$ descendants, corresponding to the extensions at different positions, and each of them is again bifurcate, and the same is true for all of their descendant, and so on, ad infinitum. This curious fact follows easily from a result of K\"{u}ndgen and Pelsmajer on \emph{nonrepetitive} colorings of \emph{outerplanar} graphs, as noted in \cite{GrytczukSZ} in a different context of the~\emph{on-line Thue games}. We will recall this short argument for completeness. It seems plausible, however, that the actual number of letters needed for such an amazing property may be much smaller.

\begin{conjecture}\label{Conjecture 5-Bifurcate Tree}
	There exists a complete bifurcate tree over an alphabet of size $5$.
\end{conjecture}
It is not hard to verify that the above conjecture is tight, as we demonstrate at the end of the following section.
\section{Results}
We shall present proofs of our results in the following subsections.

\subsection{Steady words over a $4$-letter alphabet}

For a rational $r\in[1,2]$, a \emph{fractional $r$-power} is a word $W=XP$, where $P$ is a prefix of $W$ of length $(r-1)|X|$. We say that $r$ is the \emph{exponent} of the word $W$. The famous Dejean's conjecture (now a theorem) states that the infimum of $r$ for which there exist infinite $n$-ary words without factors of exponent greater than $r$ is equal to
$$\left\{\begin{array}{ll}7/4&\mbox{ for }n=3,\\7/5&\mbox{ for }n=4,\\ n/(n-1)&\mbox{ for }n\neq3,4.\end{array}\right.$$ 
The case $n=2$ is a simple consequence of the classical theorem of Thue \cite{Thue}. The case $n=3$ was proved by Dejean \cite{Dejean1972} and the case $n=4$ by Pansiot \cite{Pansiot1984}. Cases up to $n=14$ were proved by Ollagnier \cite{Ollagnier1992} and Noori and Currie \cite{Noori2007}. Carpi \cite{Carpi2007} showed that statement holds for every $n\geq33$, while the remaining cases were proved by Currie and Rampersad \cite{Currie2011}.

\begin{theorem}\label{Theorem Steady 4}
	There exist arbitrarily long steady words over a $4$-letter alphabet.
\end{theorem}

\begin{proof}
	From Dejean's theorem we get that there exist arbitrarily long quaternary words without factors of exponent greater than $\frac75$. Let $S$ be any such word. Notice that any factor of the form $XYX$ in $S$ must satisfy $|Y|>|X|$, where $|W|$ denotes the length of a word $W$. Indeed, the oposite inequality, $|Y|\leqslant|X|$, implies that $S$ contains a factor with  exponent at least $\frac32$, but $\frac32>\frac75$.
	
	We claim that any word with the above \emph{separation} property is steady. Assume to the contrary that deleting some single interior letter in the word $S$ generates a square. We distinguish two cases corresponding to the relative position of the deleted letter:
	\begin{enumerate}
		\item $S=AC\mathtt{a}CB$ for some letter $\mathtt{a}$ and words $A,B,C$, where $C$ is non-empty. If we put $X=C$ and $Y=a$, we immediately get a contradiction with separation property of $S$, as $|C|\geqslant1$.
		\item $S=AC'\mathtt{a}C''C'C''B$ for some letter $\mathtt{a}$ and words $A,B,C',C''$ where $C'$ and $C''$ are non-empty. If $|C'|\leqslant|C''|$, then the factor $C''C'C''$ contradicts the separation property of $S$ (by putting $X=C''$ and $Y=C'$). Otherwise, we have $|C''|<|C'|$, which implies that $|\mathtt{a}C''|\leqslant|C'|$. In this case, the factor $C'(\mathtt{a}C'')C'$ contradicts the separation property of $S$ (by taking $X=C'$ and $Y=\mathtt{a}C''$).
	\end{enumerate}
Thus, the word $S$ is steady, which completes the proof.
\end{proof}

\begin{table}\centering
	{\begin{tabular}{|c||ccccccccccccccc|}\hline
			$n$&3&4&5&6&7&8&9&10&11&12&13&14&15&16&17\\
			$N$&1&2&3&5&5&7&9&12&16&21&28&37&45&58&73\\\hline
			$n$&18&19&20&21&22&23&24&25&26&27&28&29&30&31&32\\
			$N$&93&101&124&150&179&216&257&309&376&453&551&662&798&957&1149\\\hline
	\end{tabular}}
	\vskip0.5cm
	\caption{The number $N$ of quaternary steady words of length $n$ (with respect to the permutations of symbols).}\label{T1n}
\end{table}

\subsection{Steady words from lists of size $7$}In this subsections we consider the list variant of steady words.
The proof of the following theorem is inspired by a beautiful argument due to Rosenfeld \cite{Rosenfeld1}, used by him in analogous problem for square-free words.

\begin{theorem}\label{Theorem 7-list Steady}
	Let $(\mathcal{A}_1,\mathcal{A}_2,\ldots)$ be a sequence of alphabets such that $\left|\mathcal{A}_i\right|=7$ for all $i$. For every $N$ there exist at least $4^N$ steady words $w=w_1w_2\ldots w_N$, such that $w_i\in \mathcal{A}_i$ for all $i\geqslant 1$.
\end{theorem}
\begin{proof}
	Let $\mathbb{C}_n$ be the set of steady words $w=w_1w_2\ldots w_n$, such that $w_i\in \mathcal{A}_i$ for all $i$ and set $C_n$ to be the size of $\mathbb{C}_n$. Similarly, let $\mathbb{F}_{n}$ be the set of words $f=f_1f_2\ldots f_n$ with $f_i\in L_i$ for all $i$, such that $f \notin \mathbb{C}_n$ but $f_1f_2\ldots f_{n-1}$ is in $\mathbb{C}_{n-1}$. We think of $F_n$ as the number of ways we can fail by appending a new symbol at the end of a steady word of length $n-1$. Our goal is to give a good upper bound on $F_n$.
	
	\begin{claim}
		\label{claim_failsnumber}
		$$F_{n+1}\leqslant 2C_n + 2C_{n-1} + \sum_{i=0}^{\infty}\left(3+8i\right)C_{n-2-i}$$
	\end{claim}
	\begin{proof}
		Let us partition $\mathbb{F}_{n+1}$ into subsets $\mathbb{D}_1, \mathbb{D}_2, \mathbb{D}_3, \ldots$ defined such that the word $f=f_1f_2\ldots f_{n+1}$ from $\mathbb{F}_{n+1}$ is in $\mathbb{D}_j$ if the suffix of $f$ of length $2j+1$ can be reduced to a square by removing a single letter, and $f$ is not in $\mathbb{D}_{j-1}$. For example $\mathbb{D}_1$ contains words that end with $aa$ or $axa$, $\mathbb{D}_2$ contains words that end with $abxab$ (note that $abab$, $axbab$ and $abaxb$ are not possible) and $\mathbb{D}_3$ contains words that end with 
		$abxcabc$, $abcxabc$ or $abcaxbc$ (again, $abcabc$, $axbcabc$ and $abcabxc$ are not possible). It is clearly a partition by the definition of $\mathbb{F}_n$.
		
		Let $D_j$ be the size of $\mathbb{D}_{j}$. Note that $D_1\leqslant 2C_{n}$, because every word from $\mathbb{D}_1$ can be obtained by appending to some word from $\mathbb{C}_{n}$ a repetition of the last or one before last letter. We also have $D_2\leq C_{n-1}$ because, as remarked earlier, $\mathbb{D}_2$ contains only words that end with $abxab$ (where $a,b$ and $x$ are single letters) and each such word can be obtained by repeating a third- and second-from-last letter of a word from $\mathbb{C}_{n-1}$.
		
		Now consider a word $f=f_1f_2\ldots f_{n+1}$ from $\mathbb{D}_j$ for $j>2$. Since a suffix of $f$ of length $2j+1$ can be reduced to a square by removing a single non-final letter, we have that either (a) $f_{n-2j+1}f_{n-2j+1}\ldots f_{n+1} = PxQPQ$ or (b) $f_{n-2j+1}f_{n-2j+1}\ldots f_{n+1} = PQPxQ$, where $x$ is a~single letter and $P$ and $Q$ are words, where $\left|P\right|+\left|Q\right|=j$ and $Q$ is nonempty. We will count the number of words $f$ from $\mathbb{D}_j$ that fit those cases separately for every possible position of $x$ (i.e. the length of $P$). Note that in case (a) we must have $\left|P\right|\geqslant 2$, because otherwise $f$ would be contained in $\mathbb{D}_{j-1}$. Therefore, there are $j-2$ possible lengths of $P$ and for each of those lengths there are $C_{n+1-j}$ compatible words in $\mathbb{D}_j$, as each such word can be obtained from a word from $\mathbb{C}_{n+1-j}$ by appending $PQ$ at the end; this totals to $(j-1)C_{n+1-j}$.
		
		Similarly, in case (b) $\left|Q\right|\geqslant 2$, because otherwise $f$ would not be contained in $\mathbb{F}_{n+1}$. If the~length of $P$ is $0$ (respectively $1$), then there are at most $C_{n+1-j}$ (resp. $C_{n+2-j}$) compatible words in $\mathbb{D}_j$, obtained by repeating $j$ (resp. $j-1$) letters from $\mathbb{C}_{n+1-j}$ (resp. $\mathbb{C}_{n+2-j}$). If the~length of $P$ is at least $2$ (for which there are $j-3$ possibilities), then the number of compatible words is at most $7C_{n+1-j}$, as each of them can be obtained by repeating $j$ letters from a word in $\mathbb{C}_{n+1-j}$ and picking one of seven letters from the list $\mathcal{A}_{n+1-\left| P\right|}$. 
		
		After summing the above estimation in both cases, we obtain that for $j>2$
		\begin{align*}
		D_j \leqslant \left(j-1\right)C_{n+1-j} + C_{n+1-j} + C_{n+2-j} + \left(j-3\right)7C_{n+1-j} = C_{n+2-j} + \left(8j-22\right)C_{n+1-j}.
		\end{align*}
		This, together with our estimations on $D_1$ and $D_2$, implies that
		\begin{align*}
		F_{n+1}\leqslant 2C_{n} + C_{n-1} + \sum_{j=3}^{\infty}\left( C_{n+2-j} + \left(8j-22\right)C_{n+1-j}\right) = \\
		= 2C_{n} + C_{n-1} + C_{n-1} + \sum_{i=0}^{\infty}C_{n-2-i} + \sum_{i=0}^{\infty}\left(8i+2\right)C_{n-2-i}= \\
		= 2C_n + 2C_{n-1} + \sum_{i=0}^{\infty}\left(3+8i\right)C_{n-2-i},
		\end{align*}
		Which  concludes the proof of the claim.
	\end{proof}
	
	Now we inductively show that $C_n \geqslant 4 C_{n-1}$ for all $n>0$. It is clearly true for $n=1$. Note that $C_{n+1} = 7C_n - F_{n+1}$, so by Claim \ref{claim_failsnumber} we obtain that
	\begin{align*}
	C_{n+1} \geqslant 5C_n - 2C_{n-1} - \sum_{i=0}^{\infty}\left(3+8i\right)C_{n-2-i}.
	\end{align*}
	Using the induction assumption and then calculating sums of the geometric series it follows that
	\begin{align*}
	C_{n+1} \geqslant 5C_n - 2\frac{C_{n}}{4} - \sum_{i=0}^{\infty}\left(3+8i\right)\frac{C_n}{4^{2+i}} 
	= C_n \left( \frac{9}{2} - 3\sum_{i=0}^{\infty}\frac{1}{4^{2+i}} - \frac{1}{2}\sum_{i=1}^{\infty}\frac{i}{4^{i}}\right)  \\
	= C_n \left( \frac{9}{2} - \frac{1}{4} - \frac{1}{2}\sum_{i=1}^{\infty}\sum_{j=1}^{i}\frac{1}{4^{i}}\right) 
	= C_n \left( \frac{17}{4} - \frac{1}{2}\sum_{j=1}^{\infty}\sum_{i=j}^{\infty}\frac{1}{4^{i}}\right) \\
	= C_n \left( \frac{17}{4} - \frac{1}{2}\sum_{j=1}^{\infty}\frac{1}{4^{j}}\frac{4}{3}\right) 
	= C_n \left( \frac{17}{4} - \frac{2}{3}\sum_{j=1}^{\infty}\frac{1}{4^{j}}\right)  \\
	= C_n \left( \frac{17}{4} - \frac{2}{9} \right) > 4 C_n,
	\end{align*}
	which completes the proof.
\end{proof}

\subsection{Steady words are bifurcate}
As it turns out, the only square-free word which is steady, but not bifurcate is $\mathtt{1}$ over the~alphabet $\{\mathtt{1}\}$. 

\begin{theorem}\label{Theorem Steady-Bifurcate}
	Let $\mathcal{A}$ be a fixed alphabet with at least three letters. Every steady word over $\mathcal{A}$ is bifurcate over $\mathcal{A}$.
\end{theorem}
\begin{proof}
	The words $\mathtt{1}$, $\mathtt{12}$, and $\mathtt{123}$ are the only steady words with length not greater than 3. Such words are also bifurcate, therefore the theorem holds in their cases.
	
	Thus, let us assume that $n\geqslant4$ and let the square-free word $W=w_1w_2\ldots w_n$ over $\mathcal{A}$ be steady. Such word does not contain a factor $\mathtt{xyx}$, so $w_i\neq w_{i+2}$ for every $1\leqslant i\leqslant n-2$.
	
	We show that for every $0\leqslant j\leqslant n$ there exists $\mathtt{x}\in\mathcal{A}$ such that the word $$P_j(W)\mathtt{x}S_{n-j}(W)$$ is square-free. The main idea is to show that creating a palindromic factor $\mathtt{xyx}$ in the extended word is in favor of its square-freeness. Further, we use a simple fact that an extension $A\mathtt{x}B$ of a square-free word $AB$ is square-free if and only if every factor of $A\mathtt{x}B$ which contains $\mathtt{x}$ is square-free.
	\begin{case}
		[$j=0$ and $j=n$]
	\end{case}
	
	Consider the extension of the word $W=w_1w_2\cdots w_n$ by the letter $w_2$ at the beginning:$$Y=\underline{w_2}w_1w_2\ldots w_n.$$ We will show that this extension is square-free. Since $w_1w_2\ldots w_n$ is square-free, it is sufficient to show that every prefix of $Y$ is square-free. Of course, $w_2$, $w_2w_1$ and $w_2w_1w_2$ are square-free. Let us notice that $w_1\neq w_3$, so $P_4(Y)$ is square-free. Finally, $w_2w_1w_2$ is a unique factor of the form $\mathtt{xyx}$ in the word $Y$, so it cannot be a prefix of a square factor of length greater than 5 in $Y$.
	
	Analogously, one can show that the~word $w_1\ldots w_{n-1}w_n\underline{w_{n-1}}$ is square-free.
	\begin{case}
		[$j=1$ and $j=n-1$]
	\end{case}
	
	Consider the extension of $W$ by inserting the letter $w_3$ between the first and the second letter of $W$. The resulting word, $$Y=w_1\underline{w_3}w_2w_3\ldots w_n,$$ is square-free, since $w_1\neq w_3$, $w_2\neq w_4$, and $w_3w_2w_3$ is a unique factor of form $\mathtt{xyx}$. So, both words, $w_1w_3w_2w_3$ and $w_3w_2w_3$, are the prefixes of square-free factors.
	
	Analogously, one can show that the~word $$Y=w_1\ldots w_{n-2}w_{n-1}\underline{w_{n-2}}w_n$$ is square-free.
	\begin{case}
		[$2\leqslant j\leqslant n-2$]
	\end{case}
	
	In this part, let us assume that $w_j=\mathtt{a}$, $w_{j+1}=\mathtt{b}$, and $w_{j+2}=\mathtt{c}$, where $\mathtt{a}$, $\mathtt{b}$, and $\mathtt{c}$ are pairwise different letters. Thus $$W=P_{j-1}(W)\mathtt{abc}S_{n-j-2}(W).$$
	Let us investigate the extension
	$$Z=P_j(W){\mathtt{c}}S_{n-j}(W)$$
	of the word $W$, that is,
	$$Z=\lefteqn{\underbrace{\phantom{w_1w_2\ldots \mathtt{a\underline{c}bc}}}_{P_{j+3}(Z)}}w_1w_2\ldots \mathtt{a}\overbrace{\underline{\mathtt{c}}\mathtt{bc}w_{j+3}\ldots w_{n}}^{S_{n-j+1}(Z)}.$$
	We will show that $Z$ is indeed a square-free word.
	
	First notice that the word $S_{n-j+1}(Z)$ is actually an extension of the word $S_{n-j}(W)$ by inserting the second letter, $\mathtt{c}$, at the beginning. Since $S_{n-j}(W)$ is a steady word, we obtain that $S_{n-j+1}(Z)$ is square-free, by Case 1.
	
	Next notice that the word $P_{j+3}(Z)$ cannot contain a square as a suffix since it ends up with the unique palindrome $\mathtt{cbc}$. Therefore, it is sufficient to show that there are no squares in prefixes $P_{j+1}(Z)$ and $P_{j+2}(Z)$, and that no other potential square in $Z$ may contain the factor $\mathtt{cbc}$.
	
\begin{dupa}
	[No squares in $P_{j+1}(Z)$]
\end{dupa}
		
		 Let us recall that $P_j(Z)=P_j(W)$ is steady. If there is a~square suffix in $P_{j+1}(Z)$, then
		 $$P_{j+1}(Z)=AB\mathtt{c}B\mathtt{c}$$
		 for some words $A$ and $B$ (where $A$ is possibly empty) and hence $$P_{j}(Z)=AB\mathtt{c}B=P_j(W),$$ which gives us a contradiction with the assumption that $W$ is steady.		 
		 
\begin{dupa}
	[No squares in $P_{j+2}(Z)$]
\end{dupa}
		
		If there is a square suffix in $P_{j+2}(Z)$, then
		$$P_{j+2}(Z)=AB\mathtt{cb}B\mathtt{cb}$$
		for some, possibly empty, words $A$ and $B$. This construction implies that
		$$P_{j+1}(W)=AB\mathtt{cb}B\mathtt{b}.$$
		This word contains a reduction $$AB\mathtt{b}B\mathtt{b},$$ which contradicts the fact that $W$ is steady. 
		
\begin{dupa}[No squares with the factor $\mathtt{cbc}$]
\end{dupa}
		
		If $Z$ has a square $U$ which contains a unique factor $\mathtt{cbc}$, then
		$$U=\mathtt{bc}A\mathtt{c}\mathtt{bc}A\mathtt{c}\ \mbox{ or }\ U=\mathtt{c}A\mathtt{cb}\mathtt{c}A\mathtt{cb}$$
		for certain word $A$, and so $W$ has to contain a factor
		$$U_1=\mathtt{bc}A\mathtt{bc}A\mathtt{c}\ \mbox{ or }\ U_2=\mathtt{c}A\mathtt{bc}A\mathtt{cb}.$$
		Let us notice that $W$ cannot contain a factor $U_1$ since it would imply that $W$ is not square-free. Moreover, $W$ cannot contain a factor $U_2$ since it would imply that $W$ has a reduction with a factor $\mathtt{c}A\mathtt{c}A\mathtt{cb}$, which contradicts the fact that $W$ is steady.

	The proof is complete.

\end{proof}

Let us notice that the conversed statement is not true --- a word $\mathtt{12312}$ over the alphabet $\{\mathtt{1,2,3}\}$ is bifurcate, but not steady.

\subsection{Bifurcate trees of words}We will now prove the aforementioned result on bifurcate trees. Let us start with a formal definition.

A \emph{bifurcate tree} is any family $\mathbb{B}$ of bifurcate words over a fixed alphabet arranged in a~rooted tree so that the descendants of a word are its single-letter extensions at different positions. Thus, a word of length $n$ may have at most $n+1$ descendants. A bifurcate tree is called \emph{complete} if every vertex has the maximum possible number of descendants. Notice that such a tree must be infinite.

We will demonstrate that complete bifurcate trees exist over alphabets of size at least $12$. This fact follows easily from the results concerning \emph{on-line nonrepetitive games} obtained in \cite{GrytczukSZ} and independently in \cite{KeszeghZhu}, but we recall the proof for completeness. The key idea is to apply the following result form graph coloring.

A coloring $c$ the vertices of a graph $G$ is \emph{square-free} if for every simple path $v_1v_2\ldots v_n$ in $G$, the word $c(v_1)c(v_2)\cdots c(v_n)$ is square-free. A graph is \emph{planar} if it can be drawn on the plane without crossing edges. A planar graph is called \emph{outerplanar} if it has a plane drawing such that all vertices are incident to the outer face. 

\begin{theorem}[K\"{u}ndgen and Pelsmayer \cite{KundgenPelsmajer}]\label{Theorem Kundgen-Pelsmajer}Every outerplanar graph has a square-free coloring using at most $12$ colors.
\end{theorem}

	\begin{figure}[h]
	\center
	\includegraphics[width=0.75\textwidth]{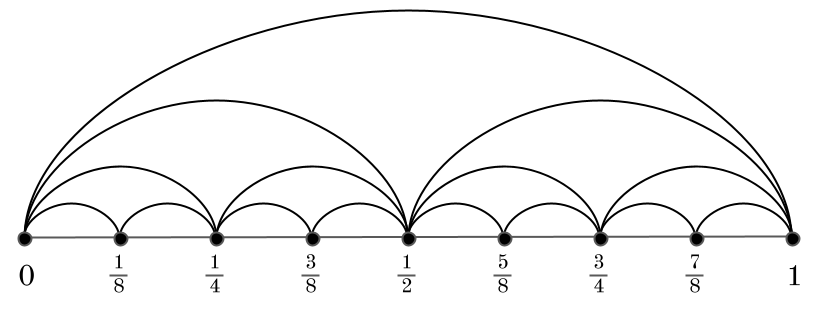}
	\caption{The graph $D_3$.}\label{Figure Dyadic}
\end{figure}

We will apply this theorem to the infinite graph $D$ constructed as follows. Let us denote $V_0=\{0,1\}$ and  $V_n=\{0,\frac{1}{2^n},\frac{2}{2^n},\frac{3}{2^n},\ldots,\frac{2^n-1}{2^n},1\}$ for $n\geqslant1$. Let $P_n$, $n\geqslant0$, denote the simple path on $V_n$ with edges joining consecutive elements in $V_n$. Let $D_n$ be the union of all paths $P_j$ with $0\leqslant j\leqslant n$ (see Figure \ref{Figure Dyadic}) and let $D$ be the countable union of all graphs $D_n$, $n\geqslant 0$.

\begin{theorem}\label{Theorem Bifurcate Tree}There exists a complete bifurcate tree over alphabet of size $12$.
\end{theorem}
\begin{proof}
 Clearly, every finite graph $D_n$ is outerplanar, hence by Theorem \ref{Theorem Kundgen-Pelsmajer} it has a square-free coloring with $12$ colors. By compactness, the infinite graph $D$ also has a $12$-coloring without a square on any simple path. Fix one such square-free coloring $c$ of the graph $D$. It is now easy to extract a complete bifurcate tree $\mathcal{B}$ out of this coloring in the following way.
 
 The root of $\mathcal{B}$ is $c(\frac{1}{2})$. Its descendants are $c(\frac{1}{4})c(\frac{1}{2})$ and $c(\frac{1}{2})c(\frac{3}{4})$. Each of these has the~following sets of descendants, $$c\left(\frac{1}{8}\right)c\left(\frac{1}{4}\right)c\left(\frac{1}{2}\right),c\left(\frac{1}{4}\right)c\left(\frac{3}{8}\right)c\left(\frac{1}{2}\right),c\left(\frac{1}{4}\right)c\left(\frac{1}{2}\right)c\left(\frac{5}{8}\right)$$ and $$c\left(\frac{3}{8}\right)c\left(\frac{1}{2}\right)c\left(\frac{3}{4}\right),c\left(\frac{1}{2}\right)c\left(\frac{5}{8}\right)c\left(\frac{3}{4}\right),c\left(\frac{1}{2}\right)c\left(\frac{3}{4}\right)c\left(\frac{7}{8}\right),$$ respectively. In general, every word $W$ in $\mathbb{B}$ corresponds to a path in the graph $D$ with vertices taken from different paths $P_n$ and closest to each other as the rational points in the~unit interval. All these words are square-free by the more general property of the coloring~$c$.
 \end{proof}
One may easily extend this result to doubly infinite words, with the notions of bifurcate words and trees extended in a natural way.
\begin{theorem}\label{Theorem Bifurcate Tree Infinite}There exists a complete bifurcate tree of doubly infinite words over alphabet of size $12$.
\end{theorem}
\begin{proof}
Consider the graph $G$ on the set of all integers which is a countable union of copies of the graph $D$ inserted in every interval $[n,n+1]$, for every integer $n$. By Theorem \ref{Theorem Kundgen-Pelsmajer} the~graph $G$ has a square-free coloring $c$ using at most $12$ colors. As the root for the constructed tree one may take the doubly infinite word$$R=\cdots c(-3)c(-2)c(-1)c(0)c(1)c(2)\cdots.$$
This word is clearly bifurcate and the whole assertion follows similarly as in the previous proof.
\end{proof}
On the other hand, it is not difficult to establish that the above results are no longer true over alphabet of size four.
\begin{theorem}\label{Theorem Bifurcate Tree Lower Bound}
	There is no complete bifurcate tree with words of length more than $6$ over a~$4$-letter alphabet.
\end{theorem}
\begin{proof}Let $\mathcal{A}=\{\mathtt{a,b,c,d}\}$ and suppose that $\mathbb{B}$ is a complete bifurcate tree over $\mathcal{A}$. We may assume that $\mathtt{ac}$ is a word in $\mathbb{B}$. Then an extension of $\mathtt{ac}$ in the middle is either $\mathtt{abc}$ or $\mathtt{adc}$. Notice that each of these words can be factorized as $XY$ such that $X$ is a word over the~alphabet $\{\mathtt{a,b}\}$ and $Y$ is a word over the alphabet $\{\mathtt{c,d}\}$. This property will be preserved in every further extension at the position separating $X$ form $Y$. So, the longest possible square-free word in $\mathbb{B}$ has the form $\mathtt{abacdc}$, which proves the assertion.
	\end{proof}
\section{Final remarks}
Let us conclude the paper with some suggestions for future research.

First notice that the assertion of Theorem \ref{Theorem Kundgen-Pelsmajer} is actually much stronger than needed for deriving conclusions on bifurcate trees. Indeed, the $12$-coloring it provides is square-free on all possible paths while for our purposes it is sufficient to consider only directed paths going always to the right. More formally, let $D^*$ denote the directed graph obtained from $D$ by orienting every edge to the right (towards the larger number).
\begin{problem}
	Determine the least possible $k$ such that there is a $k$-coloring of $D^*$ in which all directed paths are square-free.
\end{problem}
By Theorem \ref{Theorem Kundgen-Pelsmajer} we know that $k\leqslant 12$, but most probably this is not the best possible bound. Clearly, any improvement for the constant $k$ would give an improvement in statements of Theorems \ref{Theorem Bifurcate Tree} and \ref{Theorem Bifurcate Tree Infinite}. Therefore, by Theorem \ref{Theorem Bifurcate Tree Lower Bound} we know that $k\geqslant 5$.

Notice that the family of graphs $D_n$ we used in the proof of Theorem \ref{Theorem Bifurcate Tree} is actually a quite restricted subclass of planar graphs, which in turn is just one of the \emph{minor-closed} classes of graphs. It has been recently proved by Dujmović, Esperet, Joret, Walczak, and Wood \cite{DujmovicII} that every such class (except the class of all finite graphs) has bounded square-free chromatic number. In particular, every planar graph has a square-free coloring using at most $768$ colors. Perhaps these results could be used to derive other interesting properties of words. For such applications it is sufficient to restrict to \emph{oriented} planar graphs, that is, directed graphs arising from simple planar graphs by fixing for every edge one of the two possible orientations.

\begin{problem}
	Determine the least possible $k$ such that there is a $k$-coloring of any oriented planar graph in which all directed paths are square-free.
\end{problem}

Finally, let us mention of another striking connection between words and graph colorings. Let $n\geqslant 0$ be fixed, and consider all possible proper vertex $4$-colorings of the graph $D_n$. Identifying colors with letters, one may think of these colorings as of words over a $4$-letter alphabet. Let us denote this set by $\mathbb{D}_n$. Clearly, every word in $\mathbb{D}_n$ has length equal to $2^{n}+1$ --- the number of vertices in the graph $D_n$.

Let $W$ be any word of length $N$ and let $A$ be any subset of $\{1,2,\ldots,N\}$. Denote by $W_A$ the subword of $W$ along the set of indices $A$. The following statement is a simple consequence of the celebrated \emph{Four Color Theorem} (see \cite{JensenToft}, \cite{Thomas}).

\begin{theorem}
	For every pair of positive integers $n\leqslant N$ and any set of positive integers $A$, with $|A|=2^n+1$ and $\max A\leqslant 2^N+1$, there exists a word $W\in \mathbb{D}_N$ such that $W_A\in \mathbb{D}_n$.
\end{theorem}

What is more surprising is that this statement is actually equivalent to the Four Color Theorem, as proved by Descartes and Descartes \cite{Descartes} (see \cite{JensenToft}) Perhaps one could prove it directly, without refereeing to graph coloring and without huge computer verifications.

\end{document}